\documentclass[a4paper,10pt]{amsart}

\usepackage{amsthm,amssymb,amsmath,amstext,amsgen,amsbsy,amsopn,amsfonts}
\usepackage{latexsym}
\usepackage{euscript,mathrsfs,bbm} 
\usepackage{tikz,tikz-cd}
\usepackage{enumerate}  
\usepackage[shortlabels]{enumitem}
\usepackage[all]{xy}
\usepackage{placeins}
\usepackage{array,booktabs}
\usepackage{hyperref}

\usepackage{xcolor}

\usepackage{hyperref}
\hypersetup{
    colorlinks,
    linkcolor={red!50!black},
    citecolor={blue!50!black},
    urlcolor={blue!80!black}
}

\newcommand{\hyref}[2]{\hyperref[#2]{#1~\ref*{#2}}}
\newcommand{\hyrefprop}[2]{\hyperref[#2]{#1\ref*{#2}}}

\theoremstyle{plain}
\newtheorem{theorem}{Theorem}[section]

\newtheorem{lemma}[theorem]{Lemma}
\newtheorem{corollary}[theorem]{Corollary}
\newtheorem{proposition}[theorem]{Proposition}

\theoremstyle{definition}
\newtheorem{remark}[theorem]{Remark}
\newtheorem{example}[theorem]{Example}
\newtheorem{examples}[theorem]{Examples}
\newtheorem*{naive-algorithm}{Na\"ive algorithm}
\newtheorem*{refined-algorithm}{Refined algorithm}

\newtheorem{definition}[theorem]{Definition}

\newcommand{\defn}[1]{\emph{#1}}

\newcommand{\ie}{i.e.\ }

\newcommand{\etc}{etc.\ }

\newcommand{\N}{\mathbb{N}}
\newcommand{\Z}{\mathbb{Z}}
\newcommand{\Q}{\mathbb{Q}}
\newcommand{\R}{\mathbb{R}}
\newcommand{\C}{\mathbb{C}}

\newcommand{\PP}{\mathbb{P}}

\newcommand{\cat}[1]{\mathsf{#1}}

\newcommand{\mor}[2]{{\mathrm{Hom}}(#1,#2)}

\newcommand{\ext}[3]{{\mathrm{Ext}^{#1}}(#2,#3)}

\newcommand{\stab}[1]{\mathrm{Stab}(#1)}

\newcommand{\astab}[2]{\mathrm{Stab}_{#1}(#2)}
\newcommand{\astabo}[2]{\mathrm{Stab}^\circ_{#1}(#2)}

\newcommand{\Endo}[1]{\mathrm{End}(#1)}
\newcommand{\aut}[2]{\mathrm{Aut}_{#1}(#2)}
\newcommand{\auto}[2]{\mathrm{Aut}_{#1}^\mathrm{o}(#2)}

\DeclareMathOperator{\Aut}{\mathsf{Aut}}

\DeclareMathOperator{\id}{{\mathsf{id}}}

\newcommand{\spclos}[1]{\mathsf{split}\left(#1\right)}
\newcommand{\thick}[2]{\mathsf{thick}_{#1}(#2)}

\renewcommand{\phi}{\varphi}
\renewcommand{\epsilon}{\varepsilon}

\usetikzlibrary{calc,positioning,arrows,matrix}
\usepackage{ifthen}

\tikzstyle{matrix of math nodes}=[
  matrix of nodes,
  nodes={%
   execute at begin node=$,%
   execute at end node=$%
  }%
]

\title{Mass-Growth of Triangulated auto-equivalences}
\author{Jon Woolf}


\begin{document}
\begin{abstract}
We relate the mass growth (with respect to a stability condition) of an exact auto-equivalence of a triangulated category to the dynamical behaviour of its action on the space of stability conditions. One consequence is that this action is free and proper whenever the mass growth is non-vanishing.
\end{abstract}

\maketitle

\setcounter{tocdepth}{1}
\tableofcontents

\section{Introduction}

The seminal paper \cite{dhkk} initiated the dynamical study of an exact endofunctor  of a triangulated category, making a number of striking parallels with classical dynamics. These were inspired in part by the analogies between Teichm\"uller theory and Bridgeland stability conditions encapsulated in Figure \ref{analogy}. Elements of this analogy had been noted by  Kontsevich and Soibelman, and independently Seidel, before it was developed more fully by Gaiotto, Moore and Neitzke in \cite{MR3003931}. Aspects of it have been made precise by Bridgeland and Smith \cite{MR3349833} who relate the stability spaces of certain $3$-Calabi--Yau categories constructed from ideal triangulations of a surface to spaces of quadratic differentials, and separately by Haiden, Katzarkov and Kontsevich \cite{MR3735868} who identify the stability space of a certain Fukaya category of a surface with the space of marked flat structures. Beyond these precise relationships, however,  the analogy remains a useful heuristic guide.

We consider the dynamical behaviour of an exact auto-equivalence $\alpha$ of a triangulated category $\cat{D}$ acting  as an isometry of the Bridgeland metric on  the stability space $\astab{\Lambda}{\cat{D}}$. Following the analogy this should be akin to studying the action of a surface diffeomorphism  on Teichm\"uller space. There, the  mapping class group acts by isometries of the Teichm\"uller metric. This action is properly discontinuous and extends to the Thurston compactification. Analysing the fixed points of this extended action led Thurston to his classification of elements of the mapping class group as either periodic, pseudo-Anosov or reducible  \cite{MR956596}. Although we do not currently have an analogous compactification of the stability space in general --- but see \cite{bapat2020thurston} for a proposed construction ---  this suggests that the dynamics of the action on stability space should be useful for classifying automorphisms of $\cat{D}$. As a baby step in this direction we show that the infinite cyclic group generated by $\alpha$ acts freely and properly on a component of $\astab{\Lambda}{\cat{D}}$ whenever $\alpha$ has non-zero mass-growth. Similar results were obtained in the particular case of (DHKK) pseudo-Anosov functors by Kikuta \cite[\S 4]{kikuta2020curvature}.

\begin{figure}
\begin{tabular}{c|c}
Surfaces & Triangulated categories\\
\hline
\hline
Closed curve $C$ & Object $E$\\
\hline
Intersection $C \cap C'$ & Morphisms $\ext{*}{E}{E'}$\\
\hline
Flat metric & Stability condition\\
\hline
Geodesics & Stable objects\\
\hline
Length of $C$ & Mass of $E$\\
\hline
Slope of $C$ & Phase of $E$\\
\hline
Diffeomorphism & Auto-equivalence\\
\hline
Teichm\"uller space & Stability space\\
\hline
\end{tabular}
\caption{Analogies between smooth surfaces and triangulated categories.}
\label{analogy}
\end{figure}

We now explain in more detail. The \defn{entropy} $h_t(\alpha)\in \R$ of an exact endofunctor $\alpha$ of $\cat{D}$ was defined in \cite{dhkk} by analogy with the notions of entropy in dynamics.  It measures the complexity of $\alpha$. The real parameter $t$ reflects the fact that the $\Z$-grading of a triangulated category allows one to define an invariant with extra structure. The authors of \cite{dhkk} conjectured that the entropy should be related to the way the masses of objects, measured in a fixed stability condition $\sigma$ on $\cat{D}$, grow as we successively apply $\alpha$. A relationship of this kind was established by Ikeda in \cite{ikeda_2021} when the \defn{mass growth with parameter}
\[
h_{\sigma,t}(\alpha) = \sup_{E\in \cat{D}} \limsup_{n\to \infty} \frac{1}{n}\log m_{\sigma,t}(\alpha^n E) \qquad (t\in \R)
\]
of $\alpha$  was introduced. Here $m_{\sigma,t}(E) = \sum_{i=1} ^nm_\sigma(A_i)e^{\phi_i t}$ is a parameterised version of the mass of the object $E$ which takes account not only of the masses of its Harder--Narasimhan factors $A_i$ but also of their phases $\phi_i$. Ikeda shows that $h_{\sigma,t}(\alpha)$ depends only on the component of the stability space $\astab{\Lambda}{\cat{D}}$ in which $\sigma$ lies, and that  $h_t(\alpha) \geq h_{\sigma,t}(\alpha)$ with equality when $\sigma$ is in a component  containing an algebraic stability condition, \ie one whose heart is an abelian length category with finitely many isomorphism classes of simple objects.

When $\alpha$ is an auto-equivalence its action on $\astab{\Lambda}{\cat{D}}$ is controlled by the mass-growth  because $m_{\sigma,t}(\alpha^n E) =  m_{\alpha^{-n} \sigma,t}(E)$, allowing us to view $h_{\sigma,t}(\alpha)$ as a measure of how masses and phases grow as we move (backwards) along the orbit of $\alpha$ in the stability space. We show in Proposition \ref{metric bounds} that
\begin{equation}
\label{eventual displacement bound}
\lim_{n\to \infty} \frac{ d(\sigma, \alpha^n\sigma) }{ n} \geq \max\left\{ h_{\sigma,0}(\alpha) , \left| \lim_{t\to \pm \infty} \frac{ h_{\sigma,t}(\alpha) }{t}  \right|\right\}
\end{equation}
with equality when there is an algebraic stability condition in the component $\astabo{\Lambda}{\cat{D}}$ of $\sigma$. Here $d$ is the Bridgeland metric on $\astab{\Lambda}{\cat{D}}$. We refer to the term on the left of (\ref{eventual displacement bound}) as the \defn{eventual displacement} of $\alpha$. It is bounded above by the \defn{translation length} $\inf_{\sigma\in \astabo{\Lambda}{\cat{D}}} d(\sigma,\alpha\sigma)$ and depends only on the component $\astabo{\Lambda}{\cat{D}}$. The action of the infinite cyclic group generated by $\alpha$ is free and proper when the eventual displacement is strictly positive. The piecewise-linear bounds for mass growth in Proposition \ref{basic bounds} together with (\ref{eventual displacement bound}) show this occurs when $h_{\sigma,t}(\alpha)\neq 0$ for some $t\in \R$. Unfortunately, there is no simple way to compute the mass growth of a composite $\alpha \beta$ from those of $\alpha$ and $\beta$, so it  is not enough to consider a set of generators  to check whether the action of $\Aut(\cat{D})$, or some subgroup thereof, on $\astab{\Lambda}{\cat{D}}$ is free and proper.

The space $\astab{\Lambda}{\cat{D}}$ has a natural right $\C$ action which commutes with the left $\Aut(\cat{D})$ action. We also consider the induced action of $\alpha$ on the quotient $\astab{\Lambda}{\cat{D}}/\C$. This is an isometry of the induced metric  and we obtain similar inequalities. In this case the eventual displacement is bounded below by 
\[
\frac{1}{2} \left( h_{\sigma,0}(\alpha) + h_{\sigma,0}(\alpha^{-1}) \right)
\]
 in general (Lemma \ref{quotient mass bound}), and by
\[
 \frac{1}{2}\max\left\{ h_{\sigma,0}(\alpha) + h_{\sigma,0}(\alpha^{-1}) , \lim_{t\to  \infty} \frac{h_{\sigma,t}(\alpha) }{ t }-\lim_{t\to - \infty} \frac{ h_{\sigma,t}(\alpha) }{ t} \right\}
\]
when there is an algebraic stability condition in the component of $\sigma$ (Proposition~\ref{quotient eventual displacement}). This yields a  criterion for when the action on the quotient is free and proper.  The bound is sharp for (DHKK) pseudo-Anosov auto-equivalences \cite[Theorem 4.9]{kikuta2020curvature}.

We obtain these results by making direct estimates of the eventual displacement, and illustrate them by well-known examples including  pseudo-Anosov functors, auto-equivalences of semisimple categories and spherical twists. Section \ref{background} is a review of the basic definitions and results on entropy and mass growth. Section \ref{linear bounds} reviews the piecewise-linear bounds for mass growth \cite[Theorem 2.2.6]{https://doi.org/10.48550/arxiv.2008.06159}. These imply that both limits $\lim_{t\to\pm \infty} h_{\sigma,t}(\alpha) / t$ are well-defined, and can be expressed in terms of the phase growth of semistable objects or of a split generator. They also lead to criteria for when the mass growth $h_{\sigma,t}(\alpha)$ is piecewise-linear or linear in the parameter $t$. Section \ref{auto-equivalences} contains the results relating mass growth to eventual displacement.

\thanks{I would like to thank Yu-Wei Fan for explaining the connections with his, and his co-authors', work on shifting numbers \cite{https://doi.org/10.48550/arxiv.2008.06159} and suggesting that the results in \S4 should also apply to pseudo-Anosov auto-equivalences in the sense of \cite{Fan2019OnPA}. }

\section{Entropy and Mass Growth}
\label{background}
We recall the definitions of entropy, and of stability condition and mass growth, and review their key properties and inter-relationship.

\subsection{Entropy}
For an object $D\in \cat{D}$ let $\thick{}{D}$ denote the minimal thick subcategory of $\cat{D}$ containing $D$. Thus $E\in \thick{}{D}$ if there is a finite sequence of exact triangles
\begin{equation}
\label{complexity diagram}
\begin{tikzcd}
0 \ar{r} & E_1 \ar{r}\ar{d} & \cdots \ar{r} & E_{n-1} \ar{r} & E_n=E\oplus E' \ar{d}\\
& D[d_1]\ar[dashed]{ul} &&& D[d_n]\ar[dashed]{ul} &
\end{tikzcd}
\end{equation}
for some $E'\in \cat{D}$ and $d_1,\ldots, d_n\in \Z$. When $E\in \thick{}{D}$ the \defn{complexity of $E$ relative to $D$} is defined to be
\[
\delta_t(D,E) = \inf \left\{ \sum_{i=1}^n e^{d_it} \colon \ \text{diagrams (\ref{complexity diagram})} \right\};
\]
otherwise $\delta_t(D,E)=\infty$. By convention the empty sum is zero so that $\delta_t(D,0)=0$. The complexity is a well-behaved quantity, in particular:
\begin{lemma}[{\cite[Proposition 2.2]{dhkk} and \cite[Lemma 2.3]{ikeda_2021}}]
\label{complexity inequalities}
For objects $C$, $D$, $E$ and $F$ of $\cat{D}$
\begin{enumerate}
\item $\delta_t(D,F) \leq \delta_t(D,E) \delta_t(E,F) $,
\item $\delta_t(C,E) \leq \delta_t(C,D)+ \delta_t(C,F) $ whenever there is an exact triangle 
\[
D \to E \to F \to E[1],
\]
\item $\delta_t(\alpha D, \alpha E) \leq \delta_t(D,E)$ for any exact endofunctor $\alpha \colon \cat{D} \to \cat{D}$.
\end{enumerate}
(Here and in the sequel, we suppress the brackets in  $\alpha D$ \etc  to aid readability.) 
\end{lemma}

Let $\Endo{\cat{D}}$ be the ring of exact endofunctors of $\cat{D}$ and assume that $\cat{D}$ has a \defn{split-generator}, \ie an object $G$ such that $\thick{}{G} =\cat{D}$. 
\begin{definition}[{\cite[Definition 2.4]{dhkk}}]
The \defn{entropy} of  $\alpha\in \Endo{\cat{D}}$ is defined to be
\[
h_t(\alpha) = \lim_{n\to \infty} \left( \frac{1}{n}\log \delta_t(G, \alpha^nG)\right).
\]
The limit exists and is independent of the choice of generator by \cite[Lemma 2.5]{dhkk} and the entropy is a convex real-valued function of $t$ by \cite[Theorem 2.1.6]{https://doi.org/10.48550/arxiv.2008.06159}.
\end{definition}
By \cite[\S2]{dhkk}  $h_t(\alpha[d]) = h_t(\alpha)+dt$ and $h_t(\alpha^k) = kh_t(\alpha)$ for any $k\in\N_{>0}$. If $\cat{D}$ is saturated and admits a Serre functor then $h_t(\alpha)=h_t(\alpha^{-1})$ for any autoequivalence $\alpha$ \cite[Lemma  2.11]{Fan_2021}.

The entropy is conjugation-invariant:  $h_t(\alpha^{-1}\beta\alpha) = h_t(\beta)$ for an exact auto-equivalence $\alpha$ and endofunctor $\beta$ by \cite[\S2]{dhkk}. Indeed $h_t(\alpha\beta)=h_t(\beta\alpha)$ for any endomorphisms $\alpha$ and $\beta$ for which $\alpha^nG, \beta^nG\neq 0$ for $n\in \N$ \cite[Lemma 2.8]{MR3600071}. When $\alpha$ and $\beta$ commute \cite[\S2]{dhkk} gives a bound
\[
h_t(\alpha\beta)  \leq h_t(\alpha) + h_t(\beta)
\]
with equality when $\cat{D}$ is saturated with a Serre functor and $\alpha$ is an autoequivalence with $h_t(\alpha)$  odd  \cite[Lemma 2.12]{Fan_2021}. However, if $\alpha$ and $\beta$ do not commute there is no obvious relation between the entropies of $\alpha$, $\beta$ and $\alpha\beta$. 

In many contexts the entropy $h_0(\alpha)$ at $t=0$ is bounded below by the logarithm of the spectral radius of the linear endomorphism induced by $\alpha$ on Hochschild homology or on the numerical Grothendieck group --- see for example \cite[Theorem 2.8]{dhkk} and \cite[Theorem 2.13]{Kikuta2018ANO}. Often, this lower bound is sharp for auto-equivalences, but this is not always the case --- see \cite[Proposition 1.6]{Ouchi2019OnEO} for counterexamples involving K3 surfaces.

Finally \cite[Theorem 2.1.7]{https://doi.org/10.48550/arxiv.2008.06159}, see also \cite[Proposition 6.13]{Elagin2019ThreeNO}, establishes piecewise-linear upper and lower bounds 
\begin{align}
\label{PL entropy lower bound}
\tau^-(\alpha)t  \leq h_t(\alpha) \leq h_0(\alpha) + \tau^-(\alpha)t & \qquad \text{for}\ t\leq 0\ \text{and}\\ 
\label{PL entropy upper bound}
\tau^+(\alpha)t  \leq h_t(\alpha) \leq h_0(\alpha) + \tau^+(\alpha)t & \qquad \text{for}\ t\geq 0
\end{align}
for the entropy where the limits $\tau^\pm(\alpha) = \lim_{t\to \pm\infty} h_t(\alpha)/t$ exist in $\R$. These limits are known respectively as the lower and upper shifting numbers of $\alpha$. In particular note that the entropy is piecewise-linear when $h_0(\alpha)=0$.

\begin{examples}
\begin{enumerate}
\item \label{fractional CY} Suppose $\cat{D}$ is fractional Calabi-Yau of dimension $m/n \in \Q$, \ie that $\cat{D}$ has a Serre functor $S \colon \cat{D} \to \cat{D}$ such that $S^n \cong [m]$. Then by \cite[\S 2.6]{dhkk} $h_t(S) = mt/n$.
\item Let $X$ be a smooth projective variety over a field $k$, and 
\[
S = - \otimes \omega_X[\dim (X)] \colon \cat{D}^b(X) \to \cat{D}^b(X)
\]
be the Serre functor. Then by \cite[Proposition 2.12]{dhkk} $h_t(S) = \dim(X)\, t$.
\item Let  $Q$ be a quiver, and $S \colon \cat{D}^b(Q)\to \cat{D}^b(Q)$ the Serre functor on the bounded derived category of finite-dimensional representations of $Q$. When $Q$ is Dynkin $\cat{D}^b(Q)$ is fractional Calabi--Yau \cite[Theorem 3.8]{miyachi-yekutieli} and so $h_t(S)$ is linear by (\ref{fractional CY}) above. When $Q$ is not of Dynkin type
\[
h_t(S) = \log \rho\left( [S] \right) +t
\]
where  the spectral radius $\rho([S])$ of the induced endomorphism $[S]$ of the Grothendieck group satisfies $\rho([S])\geq 1$ with equality if and only if $Q$ is extended Dynkin \cite[Theorem 2.16]{dhkk}.

\item Let $f \colon X \to X$ be a regular self-map of a smooth complex projective variety $X$, and assume that the odd and even components of the induced map $H^*(f;\Q)$ on rational cohomology have different eigenvalues (with multiplicity) on the spectral radius $\rho\left( H^*(f,\Q) \right)$. Then 
\[
h_t(f^*) = \log \rho\left( H^*(f,\Q) \right)
\]
is constant \cite[Theorem 2.11]{dhkk}. If, in addition, $f$ is surjective then 
\[
h_t(f^*) = \log \rho \left( [f^*] \right) = h_\text{top}(f)
\]
 by \cite[Theorem 5.5]{10.1093/imrn/rnx131} where the latter is the \defn{topological entropy} of the map $f$, and $[f^*]$ denotes the induced endomorphism of the numerical Grothendieck group.
\item Let $\cat{D}$ be the perfect derived category of a smooth proper dg-algebra. Let $\Phi_S$ be the spherical twist about an $N$-spherical object $S\in \cat{D}$. For $t\leq 0$, we have $h_t(\Phi_S) = (1-N)t$, and if there is $0\neq E\in \cat{D}$ with $\ext{*}{E}{S} = 0$  then in addition $h_t(\Phi_S)=0$ for $t>0$ \cite[Theorem 1.4]{Ouchi2019OnEO}.
\end{enumerate}
\end{examples}
\subsection{Stability conditions}

We review the notion of stability condition on a triangulated category $\cat{D}$. Throughout we fix a surjection $\nu \colon K(\cat{D}) \to \Lambda$ from the Grothendieck group to a finite rank lattice, and a norm $|| - || \colon \Lambda \otimes \R \to \R_{\geq 0}$. Let $[E]$ denote the class of an object $E$ of $\cat{D}$ in $K(\cat{D})$.

A \defn{stability condition $\sigma = (P,Z)$} on $\cat{D}$ consists of an additive homomorphism $Z \colon \Lambda \to \C$ and a full additive subcategory $P(\phi) \subset \cat{D}$ for each $\phi \in \R$ such that
\begin{enumerate}
\item $P(\phi+1)=P(\phi)[1]$;
\item if $0\neq E\in P(\phi)$ then $Z(\nu([E])) = m(E)\exp(i\pi\phi)$ for some $m(E)\in \R_{>0}$;
\item each object $0\neq E \in \cat{D}$ admits a \defn{Harder--Narasimhan filtration}, \ie a finite collection of exact triangles
\[
\begin{tikzcd}
0 \ar{r} & E_1 \ar{r}\ar{d} & \cdots \ar{r} & E_{n-1} \ar{r} & E_n=E \ar{d}\\
& A_1\ar[dashed]{ul} &&& A_n\ar[dashed]{ul} &
\end{tikzcd}
\]
where $A_i \in P(\phi_i)$ with $\phi_1 > \cdots > \phi_n$;
\item there is a constant $K>0$ such that $m(E) \geq K|| \nu([E]) ||  $ for each $E\in P(\phi)$ and $\phi \in \R$.
\end{enumerate}
The objects of $P(\phi)$ are said to be \defn{semistable of phase $\phi$} and the quantity $m(E)$ is called the \defn{mass of $E$}. More generally, the \defn{mass of $0\neq E \in \cat{D}$} is defined in the above notation by
\[
m(E) = \sum_{i=1}^n m(A_i),
\]
 and the maximal and minimal phases are defined to be $\phi^+(E) = \phi_1$ and $\phi^-(E) = \phi_n$ respectively. These definitions make sense because the Harder--Narasimhan filtration is unique up to isomorphism. By convention $m(0)=0$ and $\phi^\pm(0)=-\infty$. The fourth condition above is called the \defn{support property} and implies that there is a strictly positive lower bound
\[
m = \inf_{0\neq E \in \cat{D}} m(E) \geq \frac{1}{K} \min \{ || \nu(\lambda)|| \colon 0\neq \lambda \in \Lambda\} >0
\]
on the masses of non-zero objects of $\cat{D}$. Since all norms on $\Lambda \otimes \R$ are equivalent the support property is independent of the particular choice of norm.

A stability condition $\sigma=(P,Z)$ determines a bounded t--structure on $\cat{D}$ with heart the extension-closure $P(0,1]$ of the collection of semistable objects with phases in the interval $(0,1]$.

Theorem 1.2 of \cite{MR2373143} implies that the set $\astab{\Lambda}{\cat{D}}$ of stability conditions on $\cat{D}$ admits the structure of a finite-dimensional complex manifold  such that the projection 
\[
\astab{\Lambda}{\cat{D}} \to \mor{\Lambda}{\C} \colon (P,Z) \mapsto Z
\]
 is a local isomorphism with respect to the linear complex structure on $\mor{\Lambda}{\C}$. The topology on $\astab{\Lambda}{\cat{D}}$ arises from the \defn{Bridgeland metric} 
 \[
d(\sigma, \tau) = \sup_{0\neq E \in \cat{D}} \max \left\{ \left| \log \frac{m_\sigma(E)}{m_\tau(E)} \right| , |\phi_\sigma^+(E)-\phi_\tau^+(E)|, |\phi_\sigma^-(E)-\phi_\tau^-(E)|  \right\}.
\]
When convenient, it suffices to take the supremum over all $\sigma$ semistable objects, and even over all such with phases in the interval $(0,1]$, when computing the metric. The space of stability conditions $\astab{\Lambda}{\cat{D}}$ has a left action 
\[
\alpha\cdot (P,Z) = (\alpha \circ P, Z\circ \alpha^{-1})
\]
of by the group $\aut{\Lambda}{\cat{D}}$ of exact auto-equivalences $\alpha$ such that the induced isomorphism $[E] \mapsto [\alpha E]$ on the Grothendieck group $K(\cat{D})$ descends to an automorphism of $\Lambda$. It also has a right action by the universal cover of $GL_2^+(\R)$  given by
\[
(P,Z)\cdot g=   (P\circ f, M^{-1}\circ Z)
\]
where we write $g=(M,f)$ as a pair consisting of a matrix $M\in GL_2^+(\R)$ and an increasing function $f \colon \R\to \R$ with $f(\phi+1)=f(\phi)+1$ such that the induced maps on $\R\PP^1 \cong \R/2\Z$ agree. This action preserves the collection of semistable objects. The universal cover $\C$ of the subgroup $\C^*\subset GL_2^+(\R)$ of dilations and rotations acts freely:  $w\in \C$ acts on charges via $Z \mapsto e^{-i\pi w}Z$, maps semistable objects of phase $\phi$ to ones of phase $\varphi-\Re(w)$, and rescales masses by $e^{\pi \Im(w)}$.

Although many examples are known, it is in general extremely difficult to compute $\astab{\Lambda}{\cat{D}}$, indeed even to show it is non-empty. However, it is conjectured that $\astab{\Lambda}{\cat{D}}$ is contractible, in particular connected, whenever it is non-empty.

Finally, we say that a stability condition is \defn{algebraic} if its heart is a finite length abelian category with finitely many isomorphism classes of simple objects. In several senses algebraic stability conditions are the simplest. Whenever $\cat{D}$ has a bounded t--structure with algebraic heart then we can construct algebraic stability conditions for $\Lambda  = K(\cat{D}) \cong \Z^n$, where $n$ is the number of isomorphism classes of simple objects in the heart, by freely assigning a charge in 
\[
\{ re^{i\pi\phi} \in \C \colon 0<r, 0< \phi \leq 1\}
\] 
to each isomorphism class of simple objects of the heart --- see \cite[Example 5.5]{MR2373143}. Moreover, if an entire component of $\astab{\Lambda}{\cat{D}}$ consists of algebraic stability conditions then that component is contractible \cite[Theorem 4.9]{MR3858773}.

\subsection{Mass growth}

It is useful to combine the masses and phases of the Harder--Narasimhan factors of an object into a single parameterised quantity. 
\begin{definition}[{\cite[\S 4.4]{dhkk} and \cite[Definition 3.1]{ikeda_2021}}]
The \defn{mass with parameter} of $E\in \cat{D}$ is 
\[
m_{\sigma,t}(E) = \sum_{i=1}^n m_\sigma(A_i)e^{\phi_i t} 
\]
where as above the $A_i \in P(\phi_i)$ are the Harder--Narasimhan factors of $E$. By convention $m_{\sigma,t}(0)=0$.
\end{definition}

The mass with parameter satisfies a triangle inequality for exact triangles, and is also closely related to the complexity.
\begin{lemma}[{\cite[Propositions 3.3 and 3.4]{ikeda_2021}}]
\label{mass inequalities}
Suppose $D\to E \to F \to D[1]$ is an exact triangle. Then $m_{\sigma,t}(E) \leq m_{\sigma,t}(D)+m_{\sigma,t}(F)$. Moreover,  
\[
m_{\sigma,t}(E)\leq m_{\sigma,t}(F)\delta_t(E,F)
\]
 for any $E,F\in \cat{D}$.
\end{lemma}

\begin{definition}[{\cite[\S1.1]{ikeda_2021}}]
The \defn{mass growth with parameter} of the exact endofunctor $\alpha \colon \cat{D} \to \cat{D}$ is 
\[
h_{\sigma,t}(\alpha) = \sup_{E\in \cat{D}}  \limsup_{n\to \infty} \left(\frac{1}{n} \log m_{\sigma,t}(\alpha^nE)\right)
\]
where, by convention, $\log(0) = -\infty$. Using the triangle inequality for mass with parameter (Lemma \ref{mass inequalities}), one obtains the same result by taking instead the supremum  over all semistable objects $E$, or even over all semistable objects with phases in the interval $(0,1]$.
\end{definition}

We recapitulate the properties of the mass growth as developed in \cite{ikeda_2021}. 
\begin{lemma}
\label{basic properties of mass growth}
The following basic properties  follow directly from the definition.
\begin{enumerate}
\item For any $k\in \N$ we have $h_{\sigma,t}(\alpha^k)=kh_{\sigma,t}(\alpha)$. 
\item For any $d\in \Z$ we have $h_{\sigma,t}(\alpha[d]) = h_{\sigma,t}(\alpha) +dt$.
\item $h_{\sigma,t}(\alpha^{-1}\beta\alpha)  = h_{\alpha\sigma,t}(\beta)$ for an exact auto-equivalence $\alpha$ and endofunctor $\beta$. In particular, if $\alpha$ preserves the component $\astabo{\Lambda}{\cat{D}}$ of $\sigma$ then $h_{\sigma,t}(\alpha^{-1}\beta\alpha)  = h_{\sigma,t}(\beta)$ by Proposition \ref{def inv}.
\end{enumerate}
\end{lemma}
 \begin{proposition}[{\cite[Proposition 3.10]{ikeda_2021}}]
 \label{def inv}
Mass growth is invariant under deformation: if $\sigma$ and $\tau$ are in the same connected component of $\astab{\Lambda}{\cat{D}}$ then $h_{\sigma,t}(\alpha) = h_{\tau,t}(\alpha)$ for each $t\in \R$.
 \end{proposition}
We now consider lower bounds for the mass growth at zero. Either $\alpha$ is \defn{object-wise nilpotent}, \ie for each $E\in \cat{D}$ there is some $n\in \N$ with  $\alpha^nE=0$, in which case $h_{\sigma,t}(\alpha)=-\infty$ for all $t\in \R$, or
\[
h_{\sigma,0}(\alpha)  \geq  \limsup_{n\to \infty} \left(\frac{1}{n} \log m_{\sigma}\right)= 0
\]
where $m_\sigma>0$ is the infimal mass of the non-zero objects. A more interesting lower bound is provided by the following result.
\begin{proposition}[{\cite[Proposition 3.11]{ikeda_2021}}]
\label{sprectral bound}
Suppose that $\sigma \in \astab{\Lambda}{\cat{D}}$ and that the endofunctor $\alpha$ induces a linear map $[\alpha] \colon \Lambda \to \Lambda$. Then $h_{\sigma,0}(\alpha) \geq \log \rho( [\alpha] )$ where $\rho( [\alpha] )$ denotes the spectral radius of  $[\alpha]$. 
\end{proposition}

\begin{example}[Gepner points]
\label{gepner1}
A stability condition $\sigma \in \astab{\Lambda}{\cat{D}}$ is a \defn{Gepner point} \cite{toda-gepner} for $\alpha\in \aut{\Lambda}{\cat{D}}$ if $\alpha \cdot \sigma = \sigma\cdot w$ for some $w \in \C$, equivalently if $\sigma\cdot \C$ is a fixed point of the action of $\alpha$ on the quotient $\astab{\Lambda}{\cat{D}}/\C$. If there is a Gepner point in the component of $\sigma$ then, without loss of generality by Proposition \ref{def inv}, we may assume $\sigma$ is that Gepner point and compute the mass growth as follows:
\begin{align*}
h_{\sigma,t}(\alpha) &= \sup_{E\in \cat{D}}  \limsup_{n\to \infty} \left(\frac{1}{n} \log m_{\sigma,t}(\alpha^nE)\right)\\
&= \sup_{E\in \cat{D}}  \limsup_{n\to \infty} \left(\frac{1}{n} \log m_{\alpha^{-n}\cdot\sigma,t}(E)\right)\\
&= \sup_{E\in \cat{D}}  \limsup_{n\to \infty} \left(\frac{1}{n} \log m_{\sigma\cdot (-nw),t}(E)\right)\\
&= \sup_{E\in \cat{D}}  \limsup_{n\to \infty} \left(\frac{1}{n} \log e^{-n\pi\Im(w) +n\Re(w)t}m_{\sigma,t}(E)\right)\\
&=-\pi\Im(w) +\Re(w)t.
\end{align*}
Since $h_{\sigma,0}(\alpha),h_{\sigma,0}(\alpha^{-1})\geq0$ we conclude that $w\in \R$ so that  $h_{\sigma,t}(\alpha)= \Re(w)t$. In particular $\alpha$ acts by rotating phases.
\end{example}

\begin{example}[Pseudo-Anosov auto-equivalences]
\label{pA1}
An auto-equivalence $\alpha$ is \defn{pseudo-Anosov} in the sense of   \cite{Fan2019OnPA} if for all $0\neq E \in \cat{D}$
\[
\limsup_{n\to \infty} \log \frac{1}{n} m_\sigma(\alpha^n E) = \log \lambda
\]
where $\lambda>1$ is the \defn{stretch factor}. Clearly this implies that $h_{\sigma,0}(\alpha)=\log \lambda$. 

Following \cite{Fan2019OnPA} we say $\alpha$ is  \defn{DHKK pseudo-Anosov} --- see \cite[Definition 4.1]{dhkk} where this is referred to simply as \defn{pseudo-Anosov} --- if there is a stability condition $\sigma$ such that $\alpha \cdot \sigma = \sigma\cdot g$ for some $g =(M,f)$ in the universal cover of $GL_2^+(\R)$ with
\[
M= \begin{pmatrix}
r & 0\\
0& 1/r
\end{pmatrix}
\qquad \text{or} \qquad
\begin{pmatrix}
1/r & 0\\
0& r
\end{pmatrix}
\]
where $|r|>1$. Every DHKK pseudo-Anosov auto-equivalence is pseudo-Anosov by \cite[Theorem 2.17]{Fan2019OnPA}. The proof uses a similar calculation to the one below, and shows that $\lambda=|r|$. Importantly however, the converse is false --- see \cite[Theorem 1.4]{Fan2019OnPA} for examples --- and the authors of \cite{Fan2019OnPA} argue that their weaker notion is a closer analogue of the classical notion of pseudo-Anosov element of the mapping class group.

The mass growth of a DHKK pseudo-Anosov auto-equivalence can be computed as follows. Firstly there is an upper bound
\begin{align*}
h_{\sigma,t}(\alpha) & = \sup_{0\neq E\ \text{semistable}} \limsup_{n\to \infty} \frac{1}{n} \log m_{\sigma,t}(\alpha^nE) \\
&= \sup_{0\neq E\ \text{semistable}} \limsup_{n\to \infty}  \frac{1}{n} \log m_{\alpha^{-n}\cdot \sigma,t}(E) \\
&= \sup_{0\neq E\ \text{semistable}} \limsup_{n\to \infty} \frac{1}{n} \log m_{ \sigma\cdot g^{-n} ,t}(E) \\
&= \sup_{0\neq E\ \text{semistable}} \limsup_{n\to \infty}\frac{1}{n} \left(  \log |M^n Z(E)| + f^{-n}(\phi(E)) t \right)\\
&\leq \log |r| + \sup_{0\leq \phi<1} \limsup_{n\to \infty} f^{-n}(\phi) t\\
&=\log|r| + f^{-1}(0) t
\end{align*}
because $f^{-1}(0)\in \Z$ from which, by induction, $nf^{-1}(0) \leq f^{-n}(\phi) < nf^{-1}(0)+1$ for any $0\leq \phi<1$. In fact we have equality above because we can use the right $\C$ action to move $\sigma$ to $\sigma'$ in the same component where $\sigma'$ has a semistable object $E$ of phase $0$ and charge $1$, or of phase $1/2$ and charge $i$ respectively, according to whether 
\[
M= \begin{pmatrix}
r & 0\\
0& 1/r
\end{pmatrix}
\quad \text{or} \quad
\begin{pmatrix}
1/r & 0\\
0& r
\end{pmatrix}.
\]
Then $h_{\sigma,t}(\alpha)  = h_{\sigma',t}(\alpha)$ by Proposition \ref{def inv}, and we still have $\alpha\cdot \sigma' = \sigma'\cdot g$ so that 
\[
h_{\sigma',t}(\alpha) \geq \limsup_{n\to \infty} \frac{1}{n} \log m_{\sigma',t}(\alpha^nE) = \log|r| +f^{-1}(0)t
\]
Finally,  $0=f(f^{-1}(0)) = f(0) + f^{-1}(0)$ because $f^{-1}(0)\in \Z$ so $f^{-1}(0)=-f(0)$. 
\end{example}
Mass growth is determined by the action on a split-generator, when one exists, and provides a lower bound for the entropy --- indeed the two coincide when $\cat{D}$ has an algebraic t--structure.
\begin{theorem}[{\cite[Theorem 3.5]{ikeda_2021}}]
\label{mass growth bounds entropy}
If $G\in \cat{D}$ is a split-generator then 
\[
\infty> h_t(\alpha)\geq h_{\sigma,t}(\alpha) = \limsup_{n\to \infty} \frac{1}{n}\log m_{\sigma,t}(\alpha^nG).
\]
Moreover, if there is an algebraic stability condition in the same component of $\astab{\Lambda}{\cat{D}}$ as $\sigma$ then $h_{\sigma,t}(\alpha) = h_t(\alpha)$. 
\end{theorem}
\begin{remark}
It is easy to see that the entropy $h_t(\id_\cat{D}) \leq 0$, but equality is not evident except for $t=0$, see the comment on page $6$ of \cite{dhkk}. However, Lemma \ref{basic properties of mass growth} and Theorem \ref{mass growth bounds entropy} imply that $h_t(\id)=0$ when $\astab{\Lambda}{\cat{D}}\neq \emptyset$,  and therefore that $h_{t}(\alpha^k)=kh_{t}(\alpha)$ for all $k\in \N$ in this case. Conversely, if $h_t(\id_\cat{D})<0$ for some $t\in \R$ then $\cat{D}$ does not admit any stability conditions.
\end{remark}

\begin{lemma}
Let $G\in \cat{D}$ be a split generator and $\alpha$ and $\beta$ commuting exact endomorphisms of $\cat{D}$.  Then for $\sigma\in \astab{\Lambda}{\cat{D}}$ there is an inequality
\[
h_{\sigma,t}(\alpha\beta)\leq h_{\sigma,t}(\alpha) + h_t(\beta).
\]
If there is an algebraic stability condition in the component $\astabo{\Lambda}{\cat{D}}$ of $\sigma$ then 
\[
h_{\sigma,t}(\alpha\beta)\leq h_{\sigma,t}(\alpha) + h_{\sigma,t}(\beta).
\]
\end{lemma}
\begin{proof}
By Theorem \ref{mass growth bounds entropy} and Lemmas \ref{complexity inequalities} and \ref{mass inequalities} we have
\begin{align*}
h_{\sigma,t}(\alpha\beta) &= \limsup_{n\to \infty} \frac{1}{n}\log m_{\sigma,t}\left( (\alpha\beta)^n G\right)
= \limsup_{n\to \infty} \frac{1}{n}\log m_{\sigma,t}\left( \alpha^n\beta^n G\right)\\
&\leq \limsup_{n\to \infty} \frac{1}{n}\log m_{\sigma,t}\left( \alpha^nG\right)\delta_t(\alpha^nG, \alpha^n\beta^nG)\\
&\leq \limsup_{n\to \infty} \frac{1}{n}\log m_{\sigma,t}\left( \alpha^nG\right)\delta_t(G, \beta^nG)\\
&\leq \limsup_{n\to \infty} \frac{1}{n}\log m_{\sigma,t}\left( \alpha^nG\right) + \limsup_{n\to \infty} \frac{1}{n}\log \delta_t(G,\beta^nG)\\
&\leq h_{\sigma,t}(\alpha)+ h_t(\beta).
\end{align*}
The last part follows directly from Theorem \ref{mass growth bounds entropy}.
\end{proof}

\section{Linear Bounds for Mass Growth and Entropy}
\label{linear bounds}

There are piecewise-linear lower and upper bounds for mass growth analogous to the bounds (\ref{PL entropy lower bound}) and (\ref{PL entropy upper bound}) for entropy. Let $\alpha$ be an exact endomorphism of $\cat{D}$, and $\sigma \in \astab{\Lambda}{\cat{D}}$ a stability condition. We assume that $\alpha$ is  not object-wise nilpotent, \ie there exists some object $E$ for which $\alpha^nE\neq 0$ for all $n\in \N$.
Define constants
\[
\phi^-_\sigma(\alpha) 
= \inf_{E\in \cat{D}}\liminf_{n\to \infty}\left( \frac{\phi_\sigma^-(\alpha^nE)}{n}\right) 
\leq \sup_{E\in \cat{D}} \limsup_{n\to \infty}\left( \frac{\phi_\sigma^+(\alpha^nE)}{n}\right) 
=\phi^+_\sigma(\alpha)
\]
with values in $[-\infty,\infty]$. Recall that by convention $\phi_\sigma^\pm(0)=-\infty$. If $\cat{D}$ has a split generator $G$ then 
\[
\phi^-_\sigma(\alpha) = \liminf_{n\to \infty}\left( \frac{ \phi_\sigma^-(\alpha^nG)}{n}\right) \qquad \text{and} \qquad \phi^+_\sigma(\alpha) = \limsup_{n\to \infty}\left( \frac{\phi_\sigma^+(\alpha^nG)}{n}\right).
\]
The next lemma gives a criterion for the existence of the limits (see below for an alternative criterion).
\begin{lemma}
\label{algebraic generator}
Suppose a component $\astabo{\Lambda}{\cat{D}}$ of the stability space contains an algebraic stability condition. Then $\cat{D}$ has a split generator $G$ such that the limits $\lim_{n\to \infty} \phi^\pm_\sigma(\alpha^nG)/n$ exist for any $\sigma \in \astabo{\Lambda}{\cat{D}}$ and  exact endofunctor $\alpha$ which is not object-wise nilpotent.
\end{lemma}
\begin{proof}
Suppose $\sigma\in \astabo{\Lambda}{\cat{D}}$ is algebraic. Without loss of generality we may assume that all objects in its heart have phase $1$. Set $G=\oplus_{i\in I} S_i[-1]$ where $\{ S_i \mid i\in I\}$ is a set representatives of the isomorphism classes of simple objects in the heart. We make this choice so that $G$ is semistable of phase $0$. Since every object in the heart has a finite composition series $G$ is a split generator for the abelian heart. Then, since every object $E\in \cat{D}$ has a finite filtration whose factors are the cohomology groups $H^*(E)$ with respect to the t--structure, $G$ is a split generator for $\cat{D}$. 

For a collection of objects $A\subset \cat{D}$ let $\spclos{A}$ be the full subcategory on those $E\in \cat{D}$ such that $E\oplus F$ is isomorphic to an object in the extension-closure of $A$ for some $F\in \cat{D}$. Since all objects have integral phases we then have
\begin{align*}
\phi^+_\sigma(E) &= \min \{ d \in \Z \mid E \in \spclos{ G[n] \colon n\leq d} \}\\
\text{and} \ \phi^-_\sigma(E) &= \max\{ d \in \Z \mid E \in \spclos{ G[n] \colon n\geq d} \}
\end{align*}
so that $E \in \spclos{ G[n] \colon \phi^-_\sigma(E) \leq n \leq \phi^+_\sigma(E) }$ for any $E\in \cat{D}$. In particular
 \begin{align*}
\alpha^{m+n} G 
&\in \spclos{ \alpha^nG[n] \colon \phi^-_\sigma(\alpha^mG) \leq n \leq \phi^+_\sigma(\alpha^mG) }\\
& \subset \spclos{ G[n] \colon \phi^-_\sigma(\alpha^mG)+\phi^-_\sigma(\alpha^nG) \leq n \leq \phi^+_\sigma(\alpha^mG) +\phi^+_\sigma(\alpha^nG) }.
\end{align*}
It follows from this and the above characterisations of $\phi^\pm_\sigma$ that $\phi^-_\sigma(\alpha^nG)$ is superadditive and  $ \phi^+_\sigma(\alpha^nG)$ subadditive in $n\in \N$. Therefore Fekete's Lemma shows that the limits
$\lim_{n\to \infty} \phi^\pm_\sigma(\alpha^nG) / n$
exist, and are given by $\sup_{n\geq 1} \phi^-_\sigma(\alpha^nG)/n$ and $\inf_{n\geq 1}  \phi^+_\sigma(\alpha^nG)/n$ respectively.
\end{proof}

The bounds in the next result were obtained by Fan and Filip \cite[Theorem 2.2.6]{https://doi.org/10.48550/arxiv.2008.06159} under the assumption that $\cat{D}$ has a split generator $G$ and admits a Serre functor. This ensures that
\[
\phi^\pm_\sigma(\alpha) = \lim_{n\to \infty}\left( \frac{\phi_\sigma^\pm(\alpha^nG)}{n} \right) = \tau^\pm(\alpha)
\]
are the shifting numbers, and so in particular are independent of $\sigma$. However, the bounds on mass growth hold without these assumptions on $\cat{D}$, as can be seen by examining their proof.
\begin{proposition}
\label{basic bounds}
Let $\alpha$ be an exact endomorphism which is not object-wise nilpotent. For any $\sigma \in \astab{\Lambda}{\cat{D}}$ there are bounds
\begin{align*}
\phi_\sigma^-(\alpha)\, t  \leq h_{\sigma,t}(\alpha)  \leq h_{\sigma,0}(\alpha)  + \phi_\sigma^-(\alpha)\, t & \qquad \text{for}\ t\leq 0 \ \text{and}\\
\phi_\sigma^+(\alpha)\, t  \leq h_{\sigma,t}(\alpha) \leq h_{\sigma,0}(\alpha)  + \phi_\sigma^+(\alpha)\, t & \qquad \text{for}\ t\geq 0.
\end{align*}
 In particular, if $h_{\sigma,0}(\alpha) =0$ then 
\[
h_{\sigma,t}(\alpha)  =
\begin{cases}
\phi_\sigma^- (\alpha)\,t & t\leq 0\\
\phi_\sigma^+(\alpha)\, t & t\geq 0
\end{cases}
\]
is piecewise-linear. Moroever, when $\cat{D}$ has a split generator and either $\cat{D}$ admits a Serre functor or there is an algebraic stability condition in the component of $\sigma$ then the lower bounds can be sharpened to
\begin{align*}
 \max\{ \phi_\sigma^-(\alpha)t,\ h_{\sigma,0}(\alpha) + \phi^+(\alpha)t \} \leq  h_{\sigma,t}(\alpha)& 
 \qquad \text{for}\ t\leq 0\ \text{and}\\
  \max\{ \phi_\sigma^+(\alpha)t,\ h_{\sigma,0}(\alpha) + \phi^-(\alpha)t \} \leq  h_{\sigma,t}(\alpha)  & 
 \qquad \text{for}\ t\geq 0.
\end{align*}
In particular, the mass growth is linear if and only if $\phi^-_\sigma(\alpha)=\phi^+_\sigma(\alpha)$ in which case
\[
h_{\sigma,t}(\alpha) = h_{\sigma,0}(\alpha) + \phi_\sigma(\alpha)t
\]
where $\phi_\sigma(\alpha)$ is the common value. Both sets of bounds are illustrated in Figure \ref{bounds figure}.
\end{proposition}
\begin{proof}
The initial bounds follow from the fact that for any $0\neq E\in \cat{D}$  we have inequalities
\begin{align*}
m_\sigma \exp\left( \phi_\sigma^-(E)\, t \right) \leq m_{\sigma,t}(E) \leq m_\sigma(E) \exp\left( \phi_\sigma^-(E)\, t\right) & \qquad \text{for}\ t\leq 0 \ \text{and}\\
m_\sigma \exp\left( \phi_\sigma^+(E)\, t \right) \leq m_{\sigma,t}(E) \leq m_\sigma(E) \exp\left( \phi_\sigma^+(E)\, t\right) &  \qquad \text{for}\ t\geq 0
\end{align*}
where $m_\sigma>0$ is the infimal mass of the non-zero objects. The last assertion follows immediately from the bounds. 

Now suppose $\cat{D}$ has a split generator and either $\cat{D}$ admits a Serre functor or there is an algebraic stability condition in the component of $\sigma$. Then by Lemma \ref{algebraic generator} or by \cite[Theorem 2.2.6]{https://doi.org/10.48550/arxiv.2008.06159} there is a split generator $G$  such that  the limits $\lim_{n\to \infty} \phi_\sigma^\pm(\alpha^nG)/n$ exist. For any $E\in \cat{D}$ there are inequalities 
\[
m_{\sigma}(E) \exp\left( \phi_\sigma^+(E)\, t \right) \leq m_{\sigma,t}(E)
\]
 for $t\leq 0$ and  $m_{\sigma}(E) \exp\left( \phi_\sigma^-(E)\, t \right) \leq m_{\sigma,t}(E)$ for $t\geq 0$. Applying these with $E=\alpha^n G$ we obtain lower bounds
\begin{align*}
h_{\sigma,0}(\alpha) + \phi^+(\alpha)t \leq  h_{\sigma,t}(\alpha) & \qquad \text{for}\ t\leq 0 \ \text{and}\\
h_{\sigma,0}(\alpha) + \phi^-(\alpha)t \leq  h_{\sigma,t}(\alpha) & \qquad \text{for}\ t\geq 0.
\end{align*}
Combining these with   the first part we obtain the sharper lower bounds. (Alternatively, when there is an algebraic stability condition in the component of $\sigma$ the sharper bounds follow from the convexity of the entropy since the mass growth and entropy agree.)

For the last part, if the mass growth is linear then (either) pair of bounds imply its slope is $\phi^-_\sigma(\alpha)=\phi^+_\sigma(\alpha)$. Conversely if $\phi^-_\sigma(\alpha)=\phi^+_\sigma(\alpha)$ then the sharper bounds imply that $h_{\sigma,t}(\alpha) = h_{\sigma,0}(\alpha) + \phi_\sigma(\alpha)t$  is linear, of slope the common value.
\end{proof}

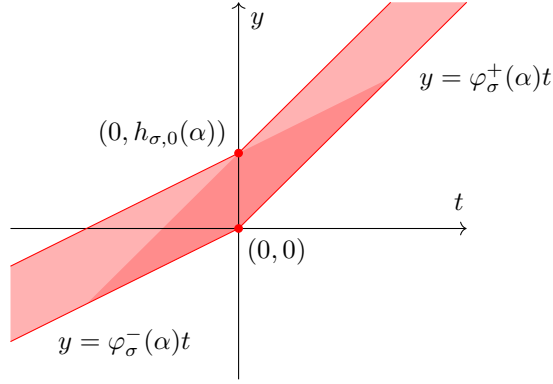
\begin{figure}

\begin{tikzpicture}
\fill[red!60, opacity=0.5] (-3,-1.5)--(-3,-0.5)--(2,2)--(0,0)--cycle;
\fill[red!60, opacity=0.5] (-2,-1)--(2,3)--(3,3)--(0,0)--cycle;
\draw[->] (-3,0)--(3,0);
\draw[->] (0,-2)--(0,3);
\node at (2.9,0.35) {$t$};
\node at (0.25,2.75) {$y$};
\draw[red] (-3,-1.5)--(0,0)--(3,3);
\draw[red] (-3,-0.5)--(0,1)--(2,3);
\node at (3.25,2) {$y=\phi_\sigma^+(\alpha) t$};
\node at (-1.5,-1.5) {$y=\phi_\sigma^-(\alpha) t$};
\node at (0.5,-0.3) {$\left(0,0\right)$};
\draw[red, fill=red] (0,0) circle (0.05);
\node at (-1,1.25) {$\left(0,h_{\sigma,0}(\alpha)\right)$};
\draw[red, fill=red] (0,1) circle (0.05);
\end{tikzpicture}
\caption{The mass-growth $y=h_{\sigma,t}(\alpha)$ lies in the red shaded region, and in the pale red shaded region when there is an algebraic stability condition in the component of $\sigma$ --- see Proposition \ref{basic bounds}.}
\label{bounds figure}
\end{figure}

\begin{remark}
These bounds are compatible with the properties in Lemma \ref{basic properties of mass growth}:
\begin{enumerate}
\item $\phi^\pm_\sigma(\alpha^k)=k\phi^\pm_\sigma(\alpha)$ and $h_{\sigma,0}(\alpha^k)=kh_{\sigma,0}(\alpha)$ for any $k\in \N$,
\item $\phi^\pm_\sigma(\alpha[d])=\phi^\pm_\sigma(\alpha)+d$ and $h_{\sigma,0}(\alpha[d])=h_{\sigma,0}(\alpha)$ for any $d\in \Z$ and
\item $\phi^\pm_\sigma(\alpha^{-1}\beta\alpha)=\phi^\pm_\sigma(\beta)$ and $h_{\sigma,0}(\alpha^{-1}\beta\alpha)=h_{\sigma,0}(\beta)$ for any exact auto-equivalence $\alpha$ and endofunctor $\beta$.
\end{enumerate}
\end{remark}

\begin{remark}[Asymptotic mass growth]
\label{asymptotics}
It follows from Proposition \ref{basic bounds} that $\lim_{t \to \pm \infty} \left( h_{\sigma,t}(\alpha) / t \right)= \phi^\pm_\sigma(\alpha)$ and hence by Proposition \ref{def inv} that these limits depend only on the component of $\sigma$ in $\astab{\Lambda}{\cat{D}}$. When $\cat{D}$ has a split generator and admits a Serre functor they are entirely independent of $\sigma$ by \cite[Theorem 2.2.6]{https://doi.org/10.48550/arxiv.2008.06159}.
\end{remark}
 \begin{example}[Serre dimensions]
 \label{serre dim}
The lower and upper Serre dimensions of a {$k$-linear} Ext-finite triangulated category $\cat{D}$ with Serre functor are defined in \cite[\S 5]{Elagin2019ThreeNO}.  When $\cat{D}$ is the perfect derived category of dg-modules over a smooth and proper dg-algebra these are the limiting slopes
\[
\underline{\textrm{Sdim}}(\cat{D}) = \lim_{t\to -\infty} h_t(S) / t \qquad \text{and} \qquad \overline{\textrm{Sdim}}(\cat{D}) = \lim_{t\to \infty} h_t(S) / t
\]
of the entropy of the Serre functor $S$  \cite[Definition 2.4]{Kikuta_2021}. Furthermore \cite[Lemma 3.11]{Kikuta_2021} and Proposition \ref{basic bounds} show that
 \[
 \lim_{t\to \pm\infty} \frac{h_t(S) }{ t} = \lim_{t\to \pm\infty} \frac{h_{\sigma,t}(S) }{ t} = \phi_\sigma^\pm(S).
 \]
 \end{example}

\begin{example}[Semisimple categories]
\label{semisimple}

Semisimple categories are a convenient class of examples in which the mass growth of all endofunctors is known, and the bounds of Proposition \ref{basic bounds} easily computed.

Let $F$ be a finite set and $\cat{D}^b(F)$ be the triangulated category of $F$-indexed bounded complexes over a field $k$. This has split-generator $k^F$ considered as a complex in degree zero. Any exact endofunctor $\alpha$ of $\cat{D}^b(F)$ is a  Fourier--Mukai functor given by a kernel $K_\alpha \in D^b(F\times F)$ with trivial differential. Such a kernel corresponds to the matrix $M_\alpha(z) \in M_{F\times F}\left( \N[z,z^{-1}] \right)$ whose entries are the Poincar\'e polynomials of the components of the kernel. The endofunctor $\alpha$ is object-wise nilpotent if and only if $M_\alpha(z) $ is nilpotent; we assume this is not the case. Composition of functors corresponds to composition of matrices. By \cite[\S2.4]{dhkk} and Theorem \ref{mass growth bounds entropy} the mass growth
\[
h_{\sigma,t}(\alpha) = h_t(\alpha) = \log \rho\left( M_\alpha(e^{-t})\right) = \lim_{n\to \infty}\frac{1}{n} \log || M_\alpha(e^{-t})^n||
\]
where $\rho$ denotes the spectral radius, and the final equality is Gelfand's formula for the spectral radius where $||\cdot||$ is (any) matrix norm. Perron--Frobenius theory gives a more explicit description of $\rho\left( M_\alpha(e^{-t})\right)$ as the maximal real eigenvalue of $M_\alpha(e^{-t})$, so that the graph is the maximal branch of the (real) spectral curve
\[
\left\{ (t,\lambda) \in \R^2\mid \det\left( M_\alpha(e^{-t}) - \lambda \right)=0 \right\},
\]
see \cite[\S2.4]{dhkk}. In particular, $h_0(\alpha) = \log \rho\left( M_\alpha(1) \right)$ is the logarithm of the maximal real eigenvalue of $M_\alpha(1)$. 

When $|F|=1$ the endofunctor is given by a Laurent polynomial $f\in \N[z,z^{-1}]$ and $h_t(\alpha) = \log f(e^{-t})$. As predicted by Proposition \ref{basic bounds} this is (piecewise) linear when $h_0(\alpha)=0$ for in this case $f(z)=z^d$ must be a monic monomial, in which case $\alpha=[-d]$ and $h_t(\alpha) = -dt$. In fact it is (piecewise) linear if and only if $h_0(\alpha)=0$. Moreover, there are examples such as $M_\alpha(z) = (z+1/z)$  for which 
\[
t \longmapsto 
\begin{cases}
\epsilon + \phi_\sigma^-(\alpha) t & t\leq 0\\
\epsilon + \phi_\sigma^+(\alpha) t & t>0
\end{cases}
\]
is not a lower bound for $h_t(\alpha) = \log (e^{-t}+e^t)$ for any $\epsilon>0$. In this sense, Proposition \ref{basic bounds} is the best possible result.

Returning to the case of arbitrary $F$, let $z^d$ be the lowest power of $z$ occurring in the entries of $M_\alpha(z)$ and set $M_\alpha(z)=z^dM'_\alpha(z)$. Then
\[
 \phi_\sigma^+(\alpha) = \lim_{t\to \infty} \frac{h_t(\alpha)}{t} 
 = \lim_{t\to \infty} \frac{\log \rho\left( M_\alpha(e^{-t}) \right)}{t}\\
 =\lim_{t\to \infty} \frac{\log \rho\left( M'_\alpha(e^{-t}) \right)}{t} -d\\
 =-d
\]
because $\rho\left( M'_\alpha(e^{-t}) \right)$ is bounded as $t\to \infty$. Similarly $ \phi_\sigma^-(\alpha) =-D$ where $z^D$ is the highest power of $z$ occurring in the entries of $M_\alpha(z)$.
\end{example}

\begin{examples}
In the following examples $h_0(\alpha)=0$ and therefore $h_{\sigma,0}(\alpha)=0$ too for any $\sigma\in \astab{\Lambda}{X}$. In these cases Proposition \ref{basic bounds} guarantees that $h_{\sigma,t}(\alpha)$ is piecewise-linear.
\label{entropy examples}
\begin{enumerate}
\item For any triangulated category $h_0([n])=0$. In this case $h_{\sigma,t}([n]) = nt$. 
\item For a Dynkin quiver $Q$ and any $\alpha \in \aut{}{\cat{D}^b(Q)}$ the entropy  $h_0(\alpha)=0$ by \cite[Corollary 2.15]{Kikuta2018ANO}. The group of auto-equivalences $\aut{}{\cat{D}^b(Q)}$ is generated by the Serre functor, shifts and auto-equivalences of the quiver $Q$ \cite[Theorem 3.8]{miyachi-yekutieli}, and it follows that $h_{\sigma,t}(\alpha)$ is always linear. 
\item\label{non-CY entropy vanishing} Let $X$ be a smooth complex projective variety with either ample canonical or ample anti-canonical bundle, and let $\alpha \in \aut{}{\cat{D}^b(X)}$. Then $h_0(\alpha)=0$ by \cite[Theorem 5.7]{10.1093/imrn/rnx131}.
\item Let $Q$ be a connected acyclic quiver with at least two vertices, and let $\cat{D}^N(Q)$ be the associated Calabi--Yau category of dimension $N\geq 2$, see \cite{KellerVandenBergh+2011+125+180}. Then
\[
h_{\sigma,t}(\Phi) = 
\begin{cases}
(1-N)t & t <0\\
0 & t\geq 0
\end{cases}
\]
where $\Phi\in\aut{}{\cat{D}^N(Q)}$ is the spherical twist about one of the simple objects in the standard heart \cite[Proposition 4.5]{ikeda_2021}.
\end{enumerate}
\end{examples}

Proposition \ref{basic bounds} implies that $\phi_\sigma^\pm(\alpha)=0$ when  $h_{\sigma,t}(\alpha)$ is constant. In the other direction, if $\phi_\sigma^\pm(\alpha)=0$ then $0 \leq h_{\sigma,t}(\alpha) \leq h_{\sigma,0}(\alpha)$. The next result provides a sufficient condition for the constancy of the mass growth with parameter, in the same spirit as \cite[Lemma 2.10]{dhkk} for the  entropy. 
\begin{lemma}
Suppose $\alpha$ preserves the heart $P(0,1]$ of the stability condition $\sigma$. Then $h_{\sigma,t}(\alpha) = h_{\sigma,0}(\alpha)$ is constant.
\end{lemma}
\begin{proof}
The existence of Harder--Narasimhan filtrations implies that any object $E$ lies in $P(-d,d]$ for some $d\in \N$. Since $\alpha$ preserves $P(0,1]$ we see that $\alpha^nE$ is also in $P(-d,d]$. Hence $m_\sigma(\alpha^nE) e^{-d|t|} \leq m_{\sigma,t}(E) \leq m_\sigma(\alpha^nE) e^{d|t|}$. The result then follows directly from the definition of $h_{\sigma,t}(\alpha)$.
\end{proof}
\begin{remark}
By the deformation invariance of the mass growth the same conclusion holds if $\alpha$ preserves the heart  of any stability condition $\tau$ in the same component of $\astab{\Lambda}{\cat{D}}$ as $\sigma$. By a simple adaptation of the argument it remains true if $\alpha$ preserves $P_\tau(I)$ for any interval $I\subset \R$ of strictly positive length.
\end{remark}

\begin{example}
Suppose $A$ is an abelian category and $\alpha \in \cat{D}^b(A)$ is induced from an exact endomorphism of $A$. Then $\alpha$ preserves the canonical heart and hence $h_{\sigma,t}(\alpha)$ is constant for any $\sigma$ in the (possibly empty) component of the space of stability conditions on $\cat{D}^b(A)$ containing those with heart $A$.
\end{example}

\section{Mass Growth of Auto-equivalences}
\label{auto-equivalences}

We relate the mass growth of an exact {\em auto-equivalence} to the properties of its action on the space of stability conditions, and on the quotient of this by $\C$.

\subsection{Translation length and eventual displacement}
We recall the definition and basic properties of the translation length and a related quantity which we refer to as the {\em eventual displacement} of an isometry.
\begin{definition}[Translation length and eventual displacement]
 Let $(X,d)$ be a metric space and $\alpha \colon X \to X$ an isometry. Then the \defn{translation length} of $\alpha$ is $l(\alpha) = \inf_{x\in X} d(x,\alpha x)$, and the \defn{eventual displacement} is 
 \[
 d(\alpha) = \lim_{n\to \infty} \frac{d(x,\alpha^nx)}{n}.
 \]
\end{definition}
\begin{lemma}
\label{eventual displacement lemma}
For an isometry $\alpha$ of a metric space $(X,d)$ the eventual displacement $d(\alpha)$ is well-defined, independent of the point $x\in X$, and satisfies $d(\alpha) \leq l(\alpha)$.
\end{lemma}
\begin{proof}
Fix $x\in X$. The triangle inequality implies that $\left(d(x,\alpha^nx)\right)$ is sub-additive. Fekete's Lemma then says that the limit of $d(x,\alpha^nx)/n$ as $n\to \infty$ exists and is given by the infimum:
\[
  \lim_{n\to \infty} \frac{d(x,\alpha^nx)}{n} = \inf_{n\geq 1}  \frac{d(x,\alpha^nx)} {n}.
\]
For any $y\in X$ and $n\in \N$ we have, again by the triangle inequality,
\[
d(x,\alpha^n x) \leq d(x,y) + d(y,\alpha^n y) + d(\alpha^n y ,\alpha^n x) = d(y,\alpha^n y) + 2d(x,y).
\]
Hence the eventual displacement $d(\alpha)$ is independent of $x\in X$. Therefore
\[
d(\alpha) = \inf_{x\in X} \inf_{n\geq 1} \frac{d(x,\alpha^nx)}{n} \leq \inf_{x\in X}  d(x,\alpha x) = l(\alpha)
\]
as claimed.
\end{proof}
It follows directly from the definition and the triangle inequality that if $\alpha$ is an invertible isometry then $d(\alpha)=d(\alpha^{-1})$, and if $\alpha$ and $\beta$ are commuting  isometries then $d(\alpha\beta) \leq d(\alpha)+d(\beta)$.
\begin{definition}
An isometry $\alpha$ of a metric space $(X,d)$ is 
\begin{enumerate}
\item \defn{elliptic} if $\alpha$ has a fixed point, \ie there is $x\in X$ with $d(x,\alpha x)=l(\alpha)=0$
\item \defn{hyperbolic} if there is $x\in X$ with  $d(x,\alpha x)=l(\alpha)>0$;
\item\defn{parabolic} if there is no $x\in X$ with  $d(x,\alpha x)=l(\alpha)$.
\end{enumerate}
\end{definition}
\begin{definition}
The action of the cyclic group $\langle \alpha\rangle$ generated by an invertible isometry $\alpha$ of a metric space $(X,d)$ is \defn{free and proper} if
\begin{enumerate}
\item each $x\in X$ has an open neighbourhood $U$ with $U\cap \alpha^n(U) \neq \emptyset \iff n=0$;
\item for each $x$  and $x'$ in distinct orbits there are open neighbourhoods $U\ni x$ and $U'\ni x'$ such that $U \cap \alpha^n(U')=\emptyset$ for all $n\in \Z$.
\end{enumerate}
\end{definition}
\begin{lemma}
\label{free and proper}
Suppose $\alpha$ is an invertible isometry with $d(\alpha)>0$. Then the action of the cyclic group $\langle \alpha \rangle$ on $(X,d)$ is free and proper. 
\end{lemma}
\begin{proof}
Fix $0<\epsilon <d(\alpha)/2$. We claim that $B_\epsilon(x)\cap \alpha^m\left( B_\epsilon(x)\right)=\emptyset$ for $m\neq 0$. Since $\alpha$ is an invertible isometry $\alpha^m\left( B_\epsilon(x)\right) = B_\epsilon(\alpha^mx)$ so it suffices to show that $d(x,\alpha^mx)\geq 2\epsilon$ for all $m\neq 0$. Suppose for a contradiction that there is some $x\in X$ and $0\neq m\in \N$ with $d(x,\alpha^m x)<2\epsilon$. Then by the triangle inequality
\[
d(x,\alpha^n x) \leq K + \left\lfloor \frac{n}{m} \right\rfloor 2\epsilon
\]
for any $n\in \N$ where $K = \max\{ d(x,\alpha^k x) \colon 0\leq k \leq m\}$. Hence
\[
d(\alpha) = \lim_{n\to \infty} \frac{d(x,\alpha^n x)}{n} \leq \frac{2\epsilon}{m} \leq 2\epsilon
\]
contradicting our choice of $\epsilon$. 

Now suppose that $x'$ is not in the orbit of $x$. Then by the above there is at most one $n\in \Z$ for which $\alpha^nx'\in B_\epsilon(x)$. Hence $\inf_{n\in \Z} d(x,\alpha^nx')=2\delta>0$ and $B_\delta(x) \cap \alpha^n\left( B_\delta(x') \right) =\emptyset$ for all $n\in \Z$.
\end{proof}

\subsection{Auto-equivalences acting on $\astab{\Lambda}{\cat{D}}$} 
We relate the mass growth of an exact auto-equivalence to the eventual displacement of the induced isometry of the space of stability conditions. Throughout we assume that $\alpha$ is in the subgroup $\auto{\Lambda}{\cat{D}} \subset \aut{\Lambda}{\cat{D}}$ of auto-equivalences which preserve a specified component $\astabo{\Lambda}{\cat{D}}$ of the stability space.

\begin{proposition}
\label{metric bounds}
Suppose $\alpha \in \auto{\Lambda}{\cat{D}}$ and $\sigma \in \astabo{\Lambda}{\cat{D}}$. Then 
\begin{equation}
\label{growth inequality}
\max\left\{ h_{\sigma,0}(\alpha) , \left| \phi^\pm_\sigma(\alpha)\right| \right\} \leq d(\alpha)
\end{equation}
with equality when $\astabo{\Lambda}{\cat{D}}$ contains an algebraic stability condition.
\end{proposition}
\begin{proof}
We estimate as follows, where the suprema are taken over $E \in P_\sigma(\phi)$ with $0<\phi\leq 1$:
\begin{align*}
d(\alpha) & = \lim_{n\to \infty} \sup_E \left\{ \left| \frac{1}{n} \log \frac{m_{\alpha^n\sigma}(E)}{m_\sigma(E)}\right|,\ \frac{ | \phi^\pm_{\alpha^n\sigma}(E) - \phi^\pm_{\sigma}(E) |}{n}   \right\} \\
& \geq \sup_E \left\{   \limsup_{n\to \infty}  \left| \frac{1}{n} \log \frac{m_{\alpha^n\sigma}(E)}{m_\sigma(E)}\right|,\    \limsup_{n\to \infty} \frac{ | \phi^\pm_{\alpha^n\sigma}(E) - \phi^\pm_{\sigma}(E) | }{n}   \right\} \\
& =  \max \left\{  \sup_E \limsup_{n\to \infty}  \left| \frac{1}{n} \log m_{\alpha^n\sigma}(E) \right|,\   \sup_E \limsup_{n\to \infty} \frac{ | \phi^\pm_{\alpha^n\sigma}(E) |}{n}   \right\} \\
& = \max \left\{ h_{\sigma,0}(\alpha^{-1}) ,\ \left| \phi^\pm_\sigma(\alpha^{-1}) \right| \right\}.
\end{align*}
The final step uses the fact that $m_{\alpha^n\sigma,0}(E)=m_{\sigma,0}(\alpha^{-n}E) \geq m_\sigma$  to remove the modulus signs on the first term, and similarly that $\phi^\pm_{\alpha^n\sigma}(E) = \phi^\pm_{\sigma}(\alpha^{-n}E)$ to obtain the second. Since $d(\alpha)=d(\alpha^{-1})$ we obtain (\ref{growth inequality}).

Now suppose that there is an algebraic stability condition in the component $\astabo{\Lambda}{\cat{D}}$ of $\sigma$. In fact, since $d(\alpha)$, $h_{\sigma,0}(\alpha)$ and $\phi_\sigma^\pm(\alpha)$ are all independent of the choice of $\sigma$ in the component, we may assume that $\sigma$ is algebraic, and even that $Z_\sigma(S)=i$ for each of the simple objects $S$ of the heart $P(0,1]$. Since each $E$ in the heart has a finite length Jordan--H\"older filtration with simple factors
\begin{align*}
\lim_{n\to \infty} \sup_E \frac{ | \phi^\pm_{\alpha^n\sigma}(E) - \phi^\pm_{\sigma}(E) | }{n} 
&\leq \limsup_{n\to \infty} \max_S \frac{ | \phi^\pm_{\alpha^n\sigma}(S) - \phi^\pm_{\sigma}(S) | }{n} \\
& =   \max_S \limsup_{n\to \infty} \frac{ | \phi^\pm_{\alpha^n\sigma}(S) - \phi^\pm_{\sigma}(S) | }{n} \\
& \leq \sup_E \limsup_{n\to \infty} \frac{ | \phi^\pm_{\alpha^n\sigma}(E) - \phi^\pm_{\sigma}(E) | }{n} = | \phi_\sigma^\pm(\alpha^{-1})|
\end{align*}
where the suprema are taken over semistable $E$ in the heart and the maxima over simple $S$ in the heart. Similarly, using the `triangle inequality'  for mass (Lemma \ref{mass inequalities}) and the fact that the choice of $\sigma$ means $E$ has precisely $m_\sigma(E)$ simple factors, we have
\begin{align*}
\lim_{n\to \infty} \sup_E \left| \frac{1}{n} \log \frac{m_{\alpha^n\sigma}(E)}{m_\sigma(E)} \right| 
& \leq \limsup_{n\to\infty} \max_S  \left| \frac{1}{n} \log \frac{m_{\sigma}(E) m_{\alpha^n\sigma}(S) }{m_\sigma(E)} \right| \\
& = \max_S \limsup_{n\to \infty} \left| \frac{1}{n} \log m_{\alpha^n\sigma}(S) \right| \\
& \leq   \sup_E \limsup_{n\to \infty} \left| \frac{1}{n} \log m_{\alpha^n\sigma}(E) \right| = h_{\sigma,0}(\alpha^{-1}).
\end{align*}
Thus $d(\alpha)\leq \max \{ h_{\sigma,0}(\alpha^{-1}), |\phi_\sigma^\pm(\alpha^{-1}) | \}$ and we have equality in (\ref{growth inequality}).
\end{proof}

\begin{corollary}
\label{fixed and free}
Suppose $\alpha\in \auto{\Lambda}{\cat{D}}$. Then
\begin{enumerate}
\item $h_{\sigma,t}(\alpha)=0$ if there is $ \tau \in \astabo{\Lambda}{\cat{D}}$ with bounded orbit under $\langle \alpha\rangle$;
\item the cyclic action of $\alpha$ is free and proper if $h_{\sigma,t}(\alpha)\neq 0$ for some $t\in \R$, and the quotient $\astabo{\Lambda}{\cat{D}} \to \langle \alpha \rangle \backslash \astabo{\Lambda}{\cat{D}}$ is a holomorphic covering of complex manifolds.
\end{enumerate}
\end{corollary}
\begin{proof}
For the first part, the existence of a bounded orbit means that $d(\alpha)=0$ and hence that $h_{\sigma,t}(\alpha)=0$ by Proposition \ref{metric bounds}. For the second part, $h_{\sigma,t}(\alpha)\neq 0$ implies $d(\alpha)>0$, again by Proposition \ref{metric bounds}, and therefore the action is free and proper by Lemma \ref{free and proper}. It follows that the quotient inherits a unique complex manifold structure such that the quotient map is a holomorphic covering.
\end{proof}

\begin{remark}
The eventual displacement  bounds the modulus of eigenvalues of the action on charges. If $\sigma=(P,Z) \in \astabo{\Lambda}{\cat{D}}$ and $\alpha\in \auto{\Lambda}{\cat{D}}$ satisfies $\alpha  Z = aZ$ for some $a\in \C$ then, taking the supremum over semistable $E$, 
\begin{align*}
d(\alpha) = \lim_{n\to \infty}\frac{d(\sigma,\alpha^n\sigma)}{n} 
& \geq \lim_{n\to\infty} \frac{1}{n} \sup_{E} \left| \log \left( \frac{m_{\alpha^n\sigma}(E)}{m_\sigma(E)}\right)\right| \\
& \geq \lim_{n\to\infty} \frac{1}{n} \sup_{E} \left| \log \left(\frac{|\alpha^n\cdot Z(E)|}{|Z(E)|}\right)\right| \\
&= \left| \log\left( |a|\right)  \right|
\end{align*}
because $m_\tau(E) \geq |W(E)|$ for any stability condition $\tau=(Q,W)$ and object $E$ of $\cat{D}$. In particular when $d(\alpha)=0$, for example when $h_{\sigma,t}(\alpha)=0$ and $\astabo{\Lambda}{\cat{D}}$ contains an algebraic stability condition, only unit complex eigenvalues can occur.  
\end{remark}

\begin{example}[Gepner points]
\label{gepner2}
Suppose $\sigma$ is a Gepner point for $\alpha\in\aut{\Lambda}{\cat{D}}$ with $\alpha\cdot \sigma = \sigma\cdot w$ for some $w\in \C$. Then $h_{\sigma,t}(\alpha)=\Re(w)t$ by Example \ref{gepner1}. Thus the cyclic action of $\alpha$ is free and proper unless $w=0$, in which case  $\sigma$ is a fixed point.
\end{example}
\begin{example}[Pseudo-Anosov auto-equivalences]
\label{pA2}
Suppose $\alpha$ is  pseudo-Anosov with stretch factor $\lambda$. Then
\[
d(\alpha) \geq h_{\sigma,0}(\alpha) = \sup_{E\in \cat{D}} \limsup_{n\to \infty} \log \frac{1}{n} m_\sigma(\alpha^nE) = \log \lambda > 0
\]
by  Proposition \ref{metric bounds}. Therefore $\langle \alpha \rangle$ acts freely and properly on $\astab{\Lambda}{\cat{D}}$.

If $\alpha$ is  DHKK pseudo-Anosov with $\alpha\cdot \sigma = \sigma\cdot g$ for some $\sigma\in \stab{\cat{D}}$ and $g=(M,f)$ in the universal cover of $GL_2^+(\R)$ then computing as in Example \ref{pA1} we obtain  (without the requirement that an algebraic stability condition exists)
\[
d(\alpha) = \max\left\{ \log|r|,|f(0)|\right\} >0
\]
where $|r|>1$ is the stretch factor. As in Example \ref{pA1}, we can always find $\sigma'$ in the orbit $\sigma\cdot \C$ with $d(\alpha) = d(\sigma',\alpha\cdot \sigma')$ so that $l(\alpha)=d(\alpha)$ too. Thus $\alpha$ is a hyperbolic isometry of $\astab{\Lambda}{\cat{D}}$. 
%
\end{example}

\begin{example}[Semisimple categories]
\label{emisimple2}
Let $F$ be a finite set and $\cat{D}^b(F)$ be the triangulated category of $F$-indexed bounded complexes over a field $k$ as in Example \ref{semisimple}. The stability space  $\stab{\cat{D}^b(F)} \cong \C^F$ is the universal cover of  the complex torus ${\C^*}^F \subset \mor{K\cat{D}^b(F)}{\C}$. The auto-equivalence $\alpha$ acts via $M_\alpha(1)$ on the charge space, and via the covering action on $\stab{\cat{D}^b(F)}$.

An endomorphism $\alpha$ represented by a matrix $M_\alpha(z) \in M_{F\times F}\left(\N[z,z^{-1}]\right)$ is an auto-equivalence if and only if $M_\alpha(1)$ is a permutation matrix and all non-zero entries of $M_\alpha(z)$ are monic monomials. This is because the only indecomposable objects in $\cat{D}^b(F)$ are the shifts of the simple objects $\{S_i \mid i\in F\}$ of the standard heart. Therefore $M_\alpha(z)^k=\text{Diag}(z^{n_i} \mid i \in F)$  is diagonal where $k\in \N$ is the order of the permutation matrix $M_\alpha(1)$. By Example \ref{semisimple} we then have
\[
h_{\sigma,t}(\alpha) = 
\begin{cases}
\min\{-n_i/k  \mid i\in F\}t & t\leq 0\\
\max\{-n_i /k \mid i\in F\}t & t\geq 0
\end{cases}
\]
so that $d(\alpha) = \max\{ |n_i|/k \mid i\in F\} $ by Proposition \ref{metric bounds}. In particular the cyclic action of $\alpha$ is free and proper unless $M_\alpha(z)^k=\id$ for some $k>0$ (in which case every orbit is periodic with period $k$).

A stability condition $\sigma$ on $\cat{D}^b(F)$ is determined by a choice of phases $\{\phi_i \mid i\in F\}$ and masses for the simple objects $S_i$. In order to minimise $d(\sigma,\alpha\sigma)$ we choose the masses to be the same. Let $\pi$ be the permutation corresponding to $M_\alpha(1)$ so that $\alpha(S_i) = S_{\pi(i)}[-m_i]$ for some $m_i\in \Z$. Note that $n_i = \sum_{j\in \mathcal{O}_i} m_j$ where $\mathcal{O}_i \subset F$ is the orbit of $i$ under $\pi$. The minimal distance between $\sigma$ and $\alpha\sigma$ occurs when we choose the phases so that 
\[
\phi_{\pi(i)}-\phi_i = m_i-\frac{n_i}{\#\mathcal{O}_i}\qquad (i\in F).
\]
This gives a minimal distance $l(\alpha)=d(\sigma,\alpha\sigma)= \max\{ |n_i| / \#\mathcal{O}_i \colon i\in F\}$. So  $l(\alpha) \geq d(\alpha)$ with equality only when $\pi$ is a cycle. We conclude that $\alpha$ is a hyperbolic isometry when $\pi$ is a cycle, and otherwise is parabolic (except when $M_\alpha(z)=\id$ in which case it is elliptic).
\end{example}
\begin{example}[Spherical twists]
\label{twists}
Recall from Examples \ref{entropy examples} that if $\Phi$ is the spherical twist about one of the simple objects in the standard heart of the $N$-Calabi--Yau category $\cat{D}^N(Q)$ associated to  a connected acyclic quiver $Q$ with at least two vertices then for $\sigma$ in the standard component
\[
h_{\sigma,t}(\Phi) = 
\begin{cases}
(1-N)t & t <0\\
0 & t\geq 0.
\end{cases}
\]
Therefore the cyclic action of $\Phi$ on $\stab{\cat{D}^N(Q)}$ is free and proper. Moreover, since the standard component contains algebraic stability conditions $d(\Phi)=N-1$ by Proposition \ref{basic bounds}. Moreover,  if $\Phi$ is the twist about the simple object $S$ then $\Phi(S)=S[N-1]$ which implies that $d(\sigma,\Phi\sigma) \geq N-1$ for $\sigma$ with heart the standard heart of $\cat{D}^N(Q)$. Thus $l(\Phi)=d(\Phi)=N-1$ and  $\Phi$ is a hyperbolic isometry of $\stab{\cat{D}^N(Q)}$.
\end{example}
\begin{example}[Serre functors]
Let $\cat{D}$ be the perfect derived category of dg-modules over a smooth and proper dg-algebra and $S$ the Serre functor. Then it follows from Example \ref{serre dim} and Proposition \ref{metric bounds} that $\max\{ -\underline{\text{Sdim}}(\cat{D}) , \overline{\text{Sdim}}(\cat{D}) \} \leq d(S)$ because $\underline{\text{Sdim}}(\cat{D}) \leq \overline{\text{Sdim}}(\cat{D})$.  Since the lower Serre dimension may be negative, see \cite[Examples 5.7 and 5.8]{Elagin2019ThreeNO}, this cannot be further simplified in general.
\end{example}

\subsection{Auto-equivalences acting on $\astab{\Lambda}{\cat{D}}/\C$} 
Suppose that $\alpha \in \auto{\Lambda}{\cat{D}}$. The action on $\astabo{\Lambda}{\cat{D}}$ descends to an action $\overline{\alpha}\cdot \overline{\sigma} = \overline{\alpha\cdot \sigma}$ on the corresponding component $\astabo{\Lambda}{\cat{D}}/\C$ of the quotient because the right action of $\auto{\Lambda}{\cat{D}}$ commutes with the left $\C$ action. Moreover, $\overline{\alpha}$ is an invertible isometry of the induced metric $d(\overline{\sigma}, \overline{\tau}) = \inf_{a\in \C} d(\sigma, \tau\cdot a)$. Clearly $d(\overline{\alpha})\leq d(\alpha)$ since $d(\overline{\sigma}, \overline{\alpha}\cdot \overline{\sigma}) \leq d(\sigma,\alpha\cdot \sigma)$.
\begin{lemma}
\label{quotient mass bound}
Suppose $\cat{D}$ has a split generator $G$, and $\alpha\in\auto{\Lambda}{\cat{D}}$. Then the eventual displacement for the component $\astabo{\Lambda}{\cat{D}}/\C$ satisfies
\[
d(\overline{\alpha})\geq \frac{h_{\sigma,0}(\alpha) + h_{\sigma,0}(\alpha^{-1})}{2}.
\]
\end{lemma}
\begin{proof}
We estimate using the mass term of the Bridgeland metric as follows:
\begin{align*}
d(\overline{\alpha})
&= \limsup_{n\to \infty} \inf_{w\in \C} \frac{d(\sigma,\alpha^n\sigma\cdot w)}{n} \\
& \geq \limsup_{n\to \infty} \inf_{w\in \C}  \sup_{E\in \cat{D}} \frac{1}{n} \left| \log \frac{m_{\alpha^n\sigma \cdot w}(E)}{m_\sigma(E)}\right|\\
& = \limsup_{n\to \infty} \inf_{w\in \C}  \sup_{E\in \cat{D}} \frac{1}{n} \left| \log \frac{m_{\alpha^n\sigma }(E)}{m_\sigma(E)} +\pi\Im(w)\right|\\
&=\limsup_{n\to \infty}  \frac{1}{2n} \left( \sup_{E\in \cat{D}} \log \frac{m_{\alpha^n\sigma }(E)}{m_\sigma(E)} - \inf_{E\in \cat{D}} \log \frac{m_{\alpha^n\sigma }(E)}{m_\sigma(E)} \right)\\
&=\limsup_{n\to \infty}  \frac{1}{2n} \left( \sup_{E\in \cat{D}} \log \frac{m_{\alpha^n\sigma }(E)}{m_\sigma(E)} + \sup_{E\in \cat{D}} \log \frac{m_\sigma(E)}{m_{\alpha^n\sigma }(E)} \right)\\
&=\limsup_{n\to \infty}  \frac{1}{2n} \left( \sup_{E\in \cat{D}} \log \frac{m_{\alpha^n\sigma }(E)}{m_\sigma(E)} + \sup_{E\in \cat{D}} \log \frac{m_{\alpha^{-n}\sigma}(E)}{m_{\sigma }(E)} \right)\\
&\geq \limsup_{n\to \infty}  \frac{1}{2n} \left( \log \frac{m_{\alpha^n\sigma }(G)}{m_\sigma(G)} + \log \frac{m_{\alpha^{-n}\sigma}(G)}{m_{\sigma }(G)} \right)\\
&=\frac{1}{2} \left( h_{\sigma,0}(\alpha^{-1}) + h_{\sigma,0}(\alpha)\right)
\end{align*}
where the infimum over $w\in \C$ is achieved by setting 
\[
\Im(w) = \frac{1}{2\pi}\left( \sup_{E\in \cat{D}} \log \frac{m_{\alpha^n\sigma }(E)}{m_\sigma(E)} + \inf_{E\in \cat{D}} \log \frac{m_{\alpha^n\sigma }(E)}{m_\sigma(E)} \right) 
\]
at the fourth step.
\end{proof}
\begin{example}[Pseudo-Anosov auto-equivalences]
\label{pA3}
When $\alpha \in \auto{\Lambda}{\cat{D}}$ is  pseudo-Anosov with stretch factor $\lambda$ then $d(\overline{\alpha})\geq \frac{1}{2}\log \lambda>0$ because $h_{\sigma,0}(\alpha)=\log \lambda$ by Example \ref{pA1} and $h_{\sigma,0}(\alpha^{-1})\geq 0$. Thus the cyclic group generated by a pseudo-Anosov auto-equivalence acts freely and properly on the quotient $\astabo{\Lambda}{\cat{D}}/\C$. 

If $\alpha$ is  DHKK  pseudo-Anosov then  $h_{\sigma,0}(\alpha) = \log  \lambda = h_{\sigma,0}(\alpha^{-1})$ by Example \ref{pA1} so that $d(\overline{\alpha})\geq \log \lambda$. This bound is sharp and the induced action on $\astabo{\Lambda}{\cat{D}}/\C$ is hyperbolic  by  \cite[Theorem 4.9]{kikuta2020curvature}.
\end{example}

\begin{proposition}
\label{quotient eventual displacement}
Suppose that the exact auto-equivalence $\alpha$ preserves a component $\astabo{\Lambda}{\cat{D}}$ of the stability space containing an algebraic stability condition. Then 
\[
\max\left \{ \frac{h_{\sigma,0}(\alpha) + h_{\sigma,0}(\alpha^{-1})}{2} , \frac{ \phi^+_\sigma(\alpha) - \phi^-_\sigma(\alpha)}{2} \right\} \leq d(\overline{\alpha})\leq \max\{ h_{\sigma,0}(\alpha) , |\phi_\sigma^\pm(\alpha)|\}.
\]
\end{proposition}
\begin{proof}
The upper bound comes directly from Proposition \ref{metric bounds} since  $d(\overline{\alpha}) \leq d(\alpha)$.

As $\astabo{\Lambda}{\cat{D}}$ contains an algebraic stability condition $\cat{D}$ has a split generator $G$ such that the limits $\lim_{n\to \infty} \phi^\pm_\sigma(\alpha^nG) /n$ exist for any $\sigma\in \astabo{\Lambda}{\cat{D}}$ by Lemma \ref{algebraic generator}. The conclusion is independent of the particular choice of $\sigma$ in the component so we may assume that $G$ is $\sigma$-semistable. Therefore $d(\overline{\alpha}) \geq \left(h_{\sigma,0}(\alpha) + h_{\sigma,0}(\alpha^{-1})\right)\!/2$ by Lemma \ref{quotient mass bound}.

Recalling that $\inf_x \sup_y f(x,y) \geq \sup_y \inf_x f(x,y)$ for any $f\colon X\times Y \to \R$ we estimate
\begin{align*}
d(\overline{\sigma},\overline{\alpha}\cdot \overline{\sigma}) 
&= \inf_{a\in \C} d(\sigma, \alpha\sigma\cdot a)\\
&= \inf_{a\in \C} \sup_E \left\{\left| \log \frac{m_{\alpha\sigma a}(E)}{m_\sigma(E)} \right| , \left| \phi^\pm_{\alpha\sigma a}(E) - \phi_{\sigma}(E)\right| \right\}\\
&\geq  \sup_E \inf_{a\in \C}  \max\left\{\left| \log \frac{m_{\alpha\sigma a}(E)}{m_\sigma(E)} \right| , \left| \phi^\pm_{\alpha\sigma a}(E) - \phi_{\sigma}(E)\right| \right\}\\
&= \sup_E \frac{1}{2}\left(\phi^+_{\alpha\sigma}(E)-\phi^-_{\alpha\sigma}(E) \right) \geq \frac{1}{2}\left(\phi^+_{\alpha\sigma}(G)-\phi^-_{\alpha\sigma}(G) \right)
\end{align*}
where the supremum is taken over all $\sigma$-semistable objects $E$ and we set 
\[
\Im(a) = -\frac{1}{\pi}\log \frac{m_{\alpha\sigma}(E)}{m_\sigma(E)} \quad \text{and} \quad \Re(a) = 2\phi_\sigma(E) - \frac{\phi^+_{\alpha\sigma}(E)+\phi^-_{\alpha\sigma}(E)}{2}
\]
to obtain the infimum before  using the fact that $G$ is semistable. Thus
\[
d(\overline{\alpha}) = \lim_{n\to \infty} \frac{d(\overline{\sigma} , \overline{\alpha}^n\cdot \overline{\sigma})}{n} 
 \geq \lim_{n\to \infty}  \frac{\phi^+_{\alpha^n\sigma}(G)-\phi^-_{\alpha^n\sigma}(G) }{2n} = \frac{\phi^+_\sigma(\alpha) - \phi^-_\sigma(\alpha)}{2}
\]
too.
\end{proof}

\begin{corollary}
\label{action on quotient}
Suppose that the exact auto-equivalence $\alpha$ preserves a component $\astabo{\Lambda}{\cat{D}}$ of the stability space containing an algebraic stability condition. Then $d(\overline{\alpha})>0$ so that the cyclic action of $\overline{\alpha}$ on $\astabo{\Lambda}{\cat{D}}/\C$ is free and proper unless $h_t(\alpha)$ is linear with $h_0(\alpha)=0=h_0(\alpha^{-1})$.
\end{corollary}
\begin{proof}
This follows  from Theorem \ref{mass growth bounds entropy}, Proposition \ref{quotient eventual displacement} and Lemma~\ref{free and proper}.
\end{proof}
\begin{remark}
The bound in Proposition \ref{quotient eventual displacement} is not sharp in general because we take the infimum over the mass and phase terms in the Bridgeland metric separately. Therefore the action of $\alpha$ on $\astabo{\Lambda}{\cat{D}}/\C$  may have strictly positive eventual displacement, and so be free and proper, even when the entropy of $\alpha$ and $\alpha^{-1}$ vanishes. I do not know of any examples of this behaviour. 
\end{remark}
\begin{remark}
If $\phi_\sigma^\pm(\alpha)=0$ then $d(\overline{\alpha}) = h_{\sigma,0}(\alpha) = h_{\sigma,0}(\alpha^{-1})$ since otherwise  the bounds of Proposition \ref{quotient eventual displacement} yield a contradiction when $\alpha$ is replaced with $\alpha^{-1}$.
\end{remark}
\begin{example}[Spherical twists]
\label{twists 2}
Recall from Example \ref{twists} that if $\Phi$ is the spherical twist about one of the simple objects in the standard heart of the $N$-Calabi--Yau category $\cat{D}^N(Q)$ associated to  a connected acyclic quiver $Q$ with at least two vertices then for $\sigma$ in the standard component
\[
h_{\sigma,t}(\Phi) = 
\begin{cases}
(1-N)t & t <0\\
0 & t\geq 0
\end{cases}
\]
and $\Phi$ acts freely and properly as a hyperbolic isometry of $\stab{\cat{D}^N(Q)}$. Since $h_{\sigma,t}(\Phi)$ is not linear the induced action on $\stab{\cat{D}^N(Q)}/\C$ is also free and proper by Corollary \ref{action on quotient}.
\end{example}

 \end{document}